\documentclass[11pt,reqno]{amsart}
\usepackage[margin=0.85in]{geometry}
\usepackage{amsmath,amssymb}
\usepackage{microtype}
\usepackage[hidelinks]{hyperref}

\newcommand{\Irr}{\operatorname{Irr}}

\numberwithin{equation}{section}
\newtheorem{theorem}{Theorem}[section]
\newtheorem{lemma}[theorem]{Lemma}

\title{Zeros and Roots of Unity for Characters of Solvable Groups}
\author{Fang Bian}
\address{Faculty of Mathematics and Statistics, Hubei University, 368 Youyi Avenue, Wuchang District, Wuhan, Hubei 430062, China}
\email{bianfang1998@hotmail.com}
\author{Yong Yang}
\address{Department of Mathematics, Texas State University, San Marcos, TX 78666, USA}
\email{yang@txstate.edu}


\begin{document}

\begin{abstract}
We prove Miller's conjectured bound for solvable groups: if $G$ is a finite solvable group and $\chi\in\Irr(G)$, then $\chi(g)$ is zero or a root of unity for at least half of the elements $g\in G$. We also prove the same bound for monomial irreducible characters of arbitrary finite groups.
\end{abstract}

\maketitle

\section{Introduction}

Let $G$ be a finite group and let $\chi\in\Irr(G)$. Define
\[
\lambda(\chi)=\frac{1}{|G|}\bigl|\{g\in G:\chi(g)=0\text{ or }\chi(g)\text{ is a root of unity}\}\bigr|.
\]
Miller conjectured that $\lambda(\chi)\geq 1/2$ for every finite group and every irreducible character \cite{Miller}. The conjecture appears as Conjecture~7.5 in Navarro's survey on characters of solvable groups \cite{NavarroSurvey}. Moret\'o and Navarro proved the bound for groups with a Sylow tower \cite{MoretoNavarro}. Navarro and Sambale later gave a counterexample for arbitrary finite groups \cite{NavarroSambale}; their example is nonsolvable. We prove that the bound holds for every finite solvable group.

\begin{theorem}\label{thm:main}
Let $G$ be a finite solvable group and let $\chi\in\Irr(G)$. Then
\[
\lambda(\chi)\geq \frac12.
\]
\end{theorem}

We use the following notation. If $H\leq G$ and $\theta$ is a character of $H$, then $\theta^G$ denotes the character induced from $H$ to $G$. We write $1_H$ for the principal character of $H$, $\overline\theta$ for the complex conjugate character, and $[\ ,\ ]_H$ for the usual inner product of class functions on $H$. Also, $\mathbb Z_{>0}$ denotes the positive integers.

An irreducible character $\beta\in\Irr(H)$ is \emph{primitive} if it is not induced from a character of a proper subgroup of $H$. It is \emph{quasiprimitive} if $\beta_N$ is homogeneous for every normal subgroup $N$ of $H$.

We first prove an induction lemma.

\section{Irreducible induction}

\begin{lemma}\label{lem:induction}
Let $H\leq G$, let $\beta\in\Irr(H)$, and suppose that $\chi=\beta^G\in\Irr(G)$. Assume that
\[
\beta(h)=0\quad\text{or}\quad |\beta(h)|^2\in\mathbb Z_{>0}
\]
for every $h\in H$. Then $\lambda(\chi)\geq 1/2$.
\end{lemma}

\begin{proof}
Set
\[
\rho=(\beta\overline\beta)^G.
\]
Let $T$ be a set of representatives for the left cosets of $H$ in $G$. The usual formula for induced character values gives, for $g\in G$,
\begin{equation}\label{eq:chi-value}
\chi(g)=\sum_{\substack{t\in T\\ t^{-1}gt\in H}}\beta(t^{-1}gt).
\end{equation}
Applying the same formula to $\rho$ gives
\begin{equation}\label{eq:rho}
\rho(g)=\sum_{\substack{t\in T\\ t^{-1}gt\in H}}|\beta(t^{-1}gt)|^2.
\end{equation}
It follows from the hypothesis on $\beta$ that $\rho(g)$ is a nonnegative integer. Frobenius reciprocity gives
\[
[\rho,1_G]_G=[\beta\overline\beta,1_H]_H=[\beta,\beta]_H=1.
\]
Since the inner product with $1_G$ is the average value of a class function, we have
\[
\sum_{g\in G}\rho(g)=|G|.
\]

For $i\geq0$, let
\[
n_i=|\{g\in G:\rho(g)=i\}|.
\]
The sets in this definition partition $G$, so
\[
|G|=\sum_{i\geq0}n_i.
\]
On the other hand, summing the values of $\rho$ according to these sets gives
\[
|G|=\sum_{g\in G}\rho(g)=\sum_{i\geq0}i n_i.
\]
Subtracting the first equality from the second yields
\[
0=-n_0+\sum_{i\geq2}(i-1)n_i,
\]
and hence
\[
n_0=\sum_{i\geq2}(i-1)n_i.
\]
Since $i-1\geq1$ for every $i\geq2$, it follows that
\[
n_0\geq\sum_{i\geq2}n_i.
\]
Therefore the number of elements for which $\rho$ is $0$ or $1$ is at least the number for which $\rho$ is at least $2$, and so
\begin{equation}\label{eq:half}
n_0+n_1\geq\frac{|G|}{2}.
\end{equation}

Suppose first that $\rho(g)=0$. Every summand in \eqref{eq:rho} is then zero, and \eqref{eq:chi-value} gives $\chi(g)=0$.

Suppose next that $\rho(g)=1$. Since every nonzero summand in \eqref{eq:rho} is a positive integer, exactly one summand is nonzero, and it is equal to $1$. Hence \eqref{eq:chi-value} has exactly one nonzero term. For some $h\in H$ conjugate in $G$ to $g$, we have
\[
\chi(g)=\beta(h),\qquad |\beta(h)|=1.
\]
We use the standard form of Kronecker's theorem stating that an algebraic integer all of whose Galois conjugates have absolute value $1$ is a root of unity. Character values are cyclotomic algebraic integers. Choose a cyclotomic field $K$ containing $\beta(h)$. Since $\operatorname{Gal}(K/\mathbb Q)$ is abelian, each $\sigma\in\operatorname{Gal}(K/\mathbb Q)$ commutes with complex conjugation. Applying $\sigma$ to
\[
\beta(h)\overline{\beta(h)}=1
\]
gives
\[
\sigma(\beta(h))\,\overline{\sigma(\beta(h))}=1.
\]
Every algebraic conjugate of $\beta(h)$ is obtained by restricting such a $\sigma$, so every algebraic conjugate has absolute value $1$. Kronecker's theorem now shows that $\beta(h)$ is a root of unity.

Every element counted by $n_0+n_1$ therefore belongs to the set defining $\lambda(\chi)$, and \eqref{eq:half} proves the lemma.
\end{proof}

Recall that an irreducible character $\chi$ of $G$ is \emph{monomial} if $\chi=\mu^G$ for a linear character $\mu$ of a subgroup of $G$.

\begin{theorem}\label{thm:monomial}
Let $G$ be a finite group and let $\chi\in\Irr(G)$ be monomial. Then
\[
\lambda(\chi)\geq \frac12.
\]
\end{theorem}

\begin{proof}
Write $\chi=\mu^G$, where $\mu$ is a linear character of a subgroup $H\leq G$. Since $|\mu(h)|=1$ for every $h\in H$, the hypotheses of Lemma~\ref{lem:induction} are satisfied.
\end{proof}

\section{Solvable groups}

We first note that primitive irreducible characters are quasiprimitive. Let $\beta\in\Irr(H)$ be primitive, and suppose that $N\triangleleft H$. If $\beta_N$ were not homogeneous, choose $\theta\in\Irr(N)$ below $\beta$. Then the inertia group $I_H(\theta)$ would be a proper subgroup of $H$, and Clifford correspondence would express $\beta$ as a character induced from $I_H(\theta)$, contrary to primitivity. Hence $\beta$ is quasiprimitive.

\begin{proof}[Proof of Theorem~\ref{thm:main}]
Choose $H\leq G$ of smallest order such that
\[
\chi=\beta^G
\]
for some $\beta\in\Irr(H)$. We claim that $\beta$ is primitive. Otherwise $\beta=\gamma^H$ for some proper subgroup $K<H$ and some $\gamma\in\Irr(K)$. By transitivity of induction,
\[
\chi=(\gamma^H)^G=\gamma^G,
\]
contrary to the minimal choice of $H$. Thus $\beta$ is primitive, and hence quasiprimitive by the preceding paragraph.

Since $G$ is solvable, so is $H$. Therefore the hypotheses of Wilde \cite[Theorem B]{Wilde} are satisfied. Hence, if $h\in H$ and $\beta(h)\neq0$, then $|\beta(h)|^2$ is a positive integer dividing $\beta(1)^2$. Thus Lemma~\ref{lem:induction} gives
\[
\lambda(\chi)\geq\frac12.
\]
\end{proof}

\section*{Acknowledgements}
The second author was supported in part by Simons Foundation Grant \#918096.

\section*{Disclosure Statement}
The authors report no conflict of interest.

\section*{Data Availability Statement}
No data were used in this work.

\end{document}